\documentclass{amsart}

\usepackage{amsfonts,amssymb,amsmath}
\usepackage{amsthm}

\def\R{\mathbb{R}}

\DeclareMathOperator\graph{graph}

\def\bdry{\partial}

\def\<{\left\langle}
\def\>{\right\rangle}

\def\ca{\tilde{\angle}}

\begin{document}

\theoremstyle{plain}
\newtheorem{thm}{Theorem}
\newtheorem{lem}[thm]{Lemma}
\newtheorem*{lm}{Lemma}
\newtheorem{corr}[thm]{Corollary}

\theoremstyle{definition}
\newtheorem{defn}[thm]{Definition}
\newtheorem{exmp}{Example}

\theoremstyle{remark}
\newtheorem{blah}[thm]{}

\title{Alexandrov curvature of convex hypersurfaces in Hilbert space}   
\author{Jonathan Dahl}         
\address{Department of Mathematics, Johns Hopkins University, Baltimore, Maryla\
nd 21218}
\email{jdahl@math.jhu.edu}

\date{}    

\begin{abstract}
  It is shown that convex hypersurfaces in Hilbert spaces have nonnegative Alexandrov curvature.
   This
extends an earlier result of Buyalo for convex hypersurfaces in Riemannian manifolds of finite
dimension.
\end{abstract}
\maketitle

\section{Introduction}

In this paper, the following result is established:

\begin{thm}\label{maintheorem} If $C$ is an open set in a Hilbert space $H$
and $\overline C$ is locally convex, then
$\bdry C$ is a nonnegatively curved Alexandrov space under the induced length metric.
\end{thm}

Questions of this sort go back to \cite{Alexandrov48}, where Alexandrov defined
Alexandrov
curvature and showed that it characterizes boundaries of locally convex bodies in $\R^3$.
This was generalized by Buyalo to the case of locally convex sets of full dimension in a Riemannian manifold in \cite{Buyalo76}.
If the ambient manifold has a positive lower bound $\kappa$ on sectional curvature, it
has also been shown in \cite{AKP08} that the convex boundary has Alexandrov
curvature $\ge\kappa$.

The proof of Theorem \ref{maintheorem} relies on approximating $\bdry C$ by smooth manifolds,
where the connection between curvature and convexity is well understood. Due to the possibly
infinite dimension of $H$, we cannot smooth by integrating over $H$ against a mollifier. As
currently known smoothing operators for infinite dimensional spaces do not preserve convexity,
we proceed by integrating over a suitably chosen finite dimensional subspace. Lemma \ref{quadruple} shows this can be done in such a way that the curvature of $\bdry C$ is controlled by the curvature of smooth, finite-dimensional approximating manifolds. A similar
approximation of infinite-dimensional curvature by finite dimensional curvature is outlined 
in \cite{Halbeisen00}.

\section{Basic definitions}

We begin by defining curvature in the sense of Alexandrov. There are several equivalent
definitions, and we will find it most convenient to work with comparison angles.

\begin{defn}For three points $x,y,z$ in a metric space $(X,d)$, the
comparison angle $\ca xyz$ is defined as
\[\ca xyz=\arccos\frac{d^2(x,y)-d^2(x,z)+d^2(y,z)}{2d(x,y)d(y,z)}.\]
\end{defn}

Recall that $(X,d)$ is called a length space if the distance between any two points equals
the infimum of the lengths of paths between them.

\begin{defn} A length space $(X,d)$ is said to have nonnegative Alexandrov curvature if
$X$ is locally complete and every $x\in X$ has a neighborhood $U_x$ which satisfies the
quadruple condition:
\[\ca bac + \ca cap +\ca pab \le 2\pi\]
for any quadruple $(a; b,c,p)$ of distinct points in $U_x$.
In this case, $X$ is called a nonnegatively curved Alexandrov space.
\end{defn}

If $X$ is a Riemannian manifold, then nonnegative Alexandrov curvature is equivalent
to nonnegative sectional curvature.

It will also be helpful to fix notation for polygonal paths.

\begin{defn}For two points $p,q$ in a vector space $V$, $\sigma_{pq}:[0,1]\to V$
denotes the constant speed linear path:
\[\sigma_{pq}(t)=(1-t)p+tq.\]
\end{defn}

\begin{defn}
A path $\tau:[0,1]\to V$ is called a polygonal path if it can be written in the form
\[\tau(t)=\sum_{i=1}^{k-1} \sigma_{p_i p_{i+1}}(kt-i) 1_{[i/k,(i+1)/k]}(t)\]
for some set of points $p_1,\dots,p_k\in V$. Here $1_A$ denotes the characteristic
function of the set $A$.
\end{defn}

\section{Approximation by smooth manifolds}

In this section, we prove two technical lemmas which allow us to approximate $C^{1,1}$
convex functions $f$ on a Hilbert space by convex functions that are smooth on a finite-dimensional linear subspace. This enables us to control the Alexandrov 
curvature of $\graph_f$, the graph of $f$ in $H\times\R$, via
the sectional curvature of the approximating smooth graphs.

\begin{lem}\label{polypath} Let $f:V\to(X,d)$ be a $\lambda$-bi-Lipschitz map from a Banach space $V$
onto a metric space $(X,d)$. For any rectifiable curve
$\sigma:[0,1]\to X$ and any $\varepsilon>0$, there exists a
polygonal path $\tau:[0,1]\to V$ such that
$f\circ\tau(0)=\sigma(0)$, $f\circ\tau(1)=\sigma(1)$, $\forall
t\in[0,1]$, $|\sigma(t)-f\circ\tau(t)|<\varepsilon$ and
$|l(\sigma)-l(f\circ\tau)| <\varepsilon$.
\end{lem}
\begin{proof}
For each rectifiable curve $\sigma_0:[0,1]\to X$ and
$\varepsilon>0$, define
\[B^1_\varepsilon(\sigma_0)
=\{\sigma:[0,1]\to X ; \forall t\in[0,1],
d(\sigma_0(t),\sigma(t))<\varepsilon,
|l(\sigma_0)-l(\sigma)|<\varepsilon\}.\] For each rectifiable
curve $\sigma_0:[0,1]\to V$ and $\varepsilon>0$, define
\[B^2_\varepsilon(\sigma_0)
=\{\sigma:[0,1]\to V ; \forall t\in[0,1],
|\sigma_0(t)-\sigma(t)|<\varepsilon,
|l(\sigma_0)-l(\sigma)|<\varepsilon\}.\]

Fix a rectifiable curve $\sigma_0:[0,1]\to V$ and $\varepsilon>0$.
For any $\sigma\in B^2_\varepsilon(\sigma_0)$, for all
$t\in[0,1]$,
\[|\sigma_0(t)-\sigma(t)|<\varepsilon\implies |f\circ\sigma_0(t)-f\circ\sigma(t)|<\lambda\varepsilon.\]
Furthermore,
\begin{equation*}
\begin{split}
|l(f\circ\sigma_0)-l(f\circ\sigma)|
&\le |l(\sigma_0)-l(\sigma)| + |l(f\circ\sigma_0)-l(\sigma_0)| + |l(f\circ\sigma)-l(\sigma)|\\
&\le \varepsilon + l(f\circ\sigma_0)+l(\sigma_0) + l(f\circ\sigma)+l(\sigma)\\
&\le \varepsilon + \lambda l(\sigma_0)+l(\sigma_0)+ \lambda l(\sigma)+l(\sigma)\\
&\le \varepsilon+(\lambda+1)l(\sigma_0)+(\lambda+1)(l(\sigma_0)+\varepsilon)\\
&\le 2(\lambda+1)(\varepsilon+l(\sigma_0)).
\end{split}
\end{equation*}
So for $\varepsilon'=2(\lambda+1)(\varepsilon+l(\sigma_0))$,
\[B^2_\varepsilon(\sigma_0)\subset f^{-1}(B^1_{\varepsilon'}(f\circ\sigma_0))\]
By a similar argument, for any rectifiable curve
$\sigma_0:[0,1]\to X$ and $\varepsilon>0$,
\[f^{-1}(B^1_\varepsilon(\sigma_0))\subset B^2_{\varepsilon'}(f^{-1}\circ\sigma_0),\]
for $\varepsilon'=2(\lambda+1)(\varepsilon+l(\sigma_0))$. Thus the
$B^2$'s and $f^{-1}(B^1)$'s determine equivalent topologies on the
space of rectifiable curves $\sigma:[0,1]\to V$. Polygonal paths
are dense under the $B^2$-topology, so they are dense under the
$f^{-1}(B^1)$-topology.
\end{proof}

\begin{lem}\label{quadruple} Let $f:\Omega\to\R$ be a $C^{1,1}$ convex function, where $\Omega$ is a domain in a
Hilbert space $H$. For any $x_0\in\Omega$, there exists $R>0$ such that $Y$, the graph of
$f$ over $B_R(x_0)$, satisfies the quadruple condition
\[\ca bac + \ca cap +\ca pab \le 2\pi\]
for any quadruple $(a; b,c,p)$ of distinct points, under the induced length metric $d$ from $H\times\R$.
\end{lem}
\begin{proof}
$f$ is convex, hence Lipschitz continuous for some Lipschitz constant $L\ge1$.
 Let $\hat f:\Omega\to \hat f(\Omega)\subset\graph_f$ be defined
by $\hat f(x)=
(x, f(x))$, and note that $\hat f$ is $\sqrt{1+L^2}$-bi-Lipschitz.
Choose $R>0$ such that $B_{3R}(x_0)\subset\Omega$.
Suppose that
$(a; b,c,p)$ is a quadruple of distinct points such that
\[\ca bac + \ca cap +\ca pab = 2\pi + \varepsilon_0 > 2\pi,\]
where $(a; b,c,p)=(\hat f(a'); \hat f(b'),\hat f(c'),\hat f(p'))$
and $a',b',c',p'\in B_R(x_0)$. The comparison angles vary continuously
in the intrinsic distances, so there exists $\varepsilon>0$ such
that if $(A;B,C,D)$ is a quadruple of points in some other metric
space $(X_1,d_1)$ with
\begin{align*}
|d(a,b)-d_1(A,B)|<\varepsilon,\quad
|d(a,c)-d_1(A,C)|&<\varepsilon,\quad
|d(a,p)-d_1(A,P)|<\varepsilon,\\
|d(b,c)-d_1(B,C)|<\varepsilon,\quad
|d(b.p)-d_1(B,P)|&<\varepsilon,\quad
|d(c,p)-d_1(C,P)|<\varepsilon,
\end{align*}
then
\[\ca BAC + \ca CAP +\ca PAB = 2\pi + (\varepsilon_0/2) > 2\pi.\]
By Lemma \ref{polypath}, we may approximate $d(a,b)$ by the length
of the image under $\hat f$ of a polygonal path $\tau_1$ determined by points
$a'=q_1,q_2,\dots,q_{k_1-1},b'=q_{k_1}\in B_{2R}(x_0)$ such that
\[d(a,b) + (\varepsilon/3) \ge
\sum_{i=1}^{k_1-1}l(\hat f\circ\sigma_{q_i q_{i+1}})
= l(\hat f\circ\tau_1)
 \ge d(a,b).
\]
Similarly, we may approximate $d(a,c)$ by 
the image under $\hat f$ of a polygonal path determined by points $a'=q_{k_1+1},q_{k_1+2},\dots,c'=q_{k_2}\in B_{2R}(x_0)$ such that
\[d(a,c) +(\varepsilon/3) \ge
\sum_{i=k_1+1}^{k_2-1}l(\hat f\circ\sigma_{q_i q_{i+1}})
 \ge d(a,c).\]
Continue in this manner choosing $q_{k_2+1},q_{k_2+2}\dots,q_{k_3},\dots,q_{k_6}$ to approximate
the remaining four intrinsic distances.

The $k_6+1$ points $q_1,\dots,q_{k_6},x_0$ lie in a $k_6$-dimensional subspace of
$H$, which we will identify as $\R^n$, $n=k_6$. Let $\varphi_\delta:\R^n\to\R$ be the standard $C^\infty$ mollifier supported on the $\delta$-ball, and define $f_\delta:B_{5R/2}(x_0)\to\R$ by $f_\delta=f\ast\varphi_\delta$, 
where the convolution occurs in the $\R^n$-variables and
$\delta<R/2$. Let $\hat f_\delta(x)=(x,f_\delta(x))$. 
As $f$ is assumed to be convex and $C^{1,1}$, it is easy to check the following properties:
\begin{enumerate}
    \item $f_\delta|_{B_{2LR}(x_0)\cap\R^n}$ is $C^\infty$.
    \item $f_\delta|_{B_{2LR}(x_0)\cap\R^n}$ is $L$-Lipschitz.
    \item $f_\delta\to f$ pointwise as $\delta\to0$.
    \item On $\R^n\cap \overline{B_{2LR}(x_0)}$, $\nabla_{\R^n}f_\delta\to \nabla_{\R^n}f$ uniformly as $\delta\to0$.
    \item For every rectifiable curve $\sigma:[0,1]\to B_{2R}(x_0)$, 
$l(\hat f_\delta\circ\sigma)
    \to l(\hat f\circ\sigma)$. This convergence is uniform on sets
    $\{\sigma:[0,1]\to B_{2R}(x_0) ; l(\sigma)<C\}$ with $C\in\R$.
    \item $f_\delta$ is convex.
\end{enumerate}
Let
$Y_\delta$ denote the graph of $f_\delta$ over $B_{2R}(x_0)$
with metric $d_\delta$ induced by $H\times\R$, and let
$Y_{\delta,n}$ denote the graph of $f_\delta$ over
$B_{2R}(x_0)\cap\R^n$ with metric $d_{\delta,n}$ induced by
$\R^n\times\R$. Note that $f_\delta|_{B_{2R}(x_0)\cap\R^n}$ is a
$C^\infty$ convex function over a domain in $\R^n$, so
$Y_{\delta,n}$ is a Riemannian manifold of nonnegative sectional
curvature. In particular, it satisfies the quadruple condition. We
will obtain a contradiction by showing
\begin{align*}
|d(a,b)-d_{\delta,n}(\hat f_{\delta}(a'),\hat
f_{\delta}(b'))|&<\varepsilon,\quad
|d(a,c)-d_{\delta,n}(\hat f_{\delta}(a'),\hat f_{\delta}(c'))|<\varepsilon,\\
|d(a,p)-d_{\delta,n}(\hat f_{\delta}(a'),\hat
f_{\delta}(p'))|&<\varepsilon,\quad
|d(b,c)-d_{\delta,n}(\hat f_{\delta}(b'),\hat f_{\delta}(c'))|<\varepsilon,\\
|d(b,p)-d_{\delta,n}(\hat f_{\delta}(b'),\hat
f_{\delta}(p'))|&<\varepsilon,\quad |d(c,p)-d_{\delta,n}(\hat
f_{\delta}(c'),\hat f_{\delta}(p'))|<\varepsilon.
\end{align*}

Let $C=d(a,b)+d(a,c)+\dots+d(c,p)+\varepsilon$. Choosing
$\delta_0$ small with respect to $C$, we have for all
$\delta<\delta_0$,
\[\tau\in\{\sigma:[0,1]\to B_{2R}(x_0) ; l(\sigma)<C\}
\implies |l(\hat f_\delta\circ\tau)-l(\hat f\circ\tau)|\le
\varepsilon/3.\]

Recall that $\tau_1$ is the polygonal path determined by $q_1,\dots,q_{k_1}$.
\[l(\tau_1)\le l(\hat f\circ\tau_1)
=\sum_{i=1}^{k_1-1}l(\hat f\circ\sigma_{q_i q_{i+1}}) \le
d(a,b)+(\varepsilon/3)<C,\] so $l(\hat f_\delta\circ\tau_1)\le
l(\hat f\circ\tau_1)+(\varepsilon/3)$ for $\delta<\delta_0$. $\hat
f_\delta\circ\tau_1:[0,1]\to Y_{\delta,n}$ is a path from $\hat
f_\delta(a')$ to $\hat f_\delta(b')$, so
\[d_{\delta,n}(\hat f_{\delta}(a'),\hat f_{\delta}(b'))
\le l(\hat f_\delta\circ\tau_1) \le l(\hat
f\circ\tau_1)+(\varepsilon/3) \le d(a,b) + (2\varepsilon/3).\]


Applying Lemma \ref{polypath} again, choose $\tau_2:[0,1]\to
B_{2R}(x_0)\cap\R^n$ such that
\[d_{\delta,n}(\hat f_{\delta}(a'),\hat f_{\delta}(b'))
\ge l(\hat f_\delta\circ\tau_2)-(\varepsilon/6).\] Note that
\[l(\tau_2)\le l(\hat f_\delta\circ\tau_2) \le
d_{\delta,n}(\hat f_{\delta}(a'),\hat f_{\delta}(b')) +
(\varepsilon/6) \le d(a,b)+(5\varepsilon/6) < C.\] For
$\delta<\delta_0$,
\[l(\hat f_\delta\circ\tau_2)\ge l(\hat f\circ\tau_2) - (\varepsilon/3),\]
so
\[d_{\delta,n}(\hat f_{\delta}(a'),\hat f_{\delta}(b'))
> l(\hat f\circ\tau_2)-\varepsilon
\ge d(a,b)-\varepsilon.\] The remaining inequalities follow in a
similar manner, for the same choice of $C$ and $\delta_0$. So for
$\delta<\delta_0$, the quadruple $(\hat f_{\delta}(a');\hat
f_{\delta}(b'), \hat f_{\delta}(c'),\hat f_{\delta}(p'))$ violates
the quadruple condition in the Riemannian manifold of nonnegative
sectional curvature $Y_{\delta, n}$. Therefore our original
assumption is false and $Y$ satisfies the quadruple condition.
\end{proof}

\section{Proof of Theorem \ref{maintheorem}}
\begin{proof}[Proof of Theorem \ref{maintheorem}]
We must prove the quadruple condition holds in a neighborhood of
every $x_0\in\bdry C$. Let $C'=B_{2\rho}(x_0)\cap C$, where $\rho$
is chosen small enough to make $C'$ convex. Note that the
intrinsic balls of radius $\rho$ about $x_0$ are the same for $C$
and $C'$. Choose a point $y\in C'$, and $r\in(0,\rho/2)$ such that
$B_{2r}(y)\subset C'$. Let $H'$ be the hyperplane through $x_0$
with normal vector $y-x_0$. For any $x\in H'\cap B_{2r}(x_0)$, let
$L_x$ be the line through $x$ spanned by $y-x_0$. $L_x\cap C'$ is
convex and $C'$ is open and bounded, so $L_x\cap C'$ is a bounded
interval. $x+(y-x_0)\in L_x\cap C'$, so $L_x\cap C'\neq\emptyset$.
Considering $y-x_0$ as the upward direction, let $f(x)$ denote the
$\R$-coordinate of the bottom endpoint of $L_x\cap C'$ in
$H'\times \R$. $f:H'\cap B_{2r}(x_0)\to\R$ is then a convex
function, as the epigraph is convex. Furthermore, the graph of $f$
is a neighborhood of $x_0$ in $\bdry C'$, and thus also in $\bdry
C$ since $2r<\rho$.

$f$ is convex, hence Lipschitz continuous for some Lipschitz constant $L\ge1$.
As shown in \cite{LL86}, for all small enough $\varepsilon>0$, the
inf-sup-convolution
\[g_\varepsilon(x)=\inf_{z\in H'\cap B_{2r}(x_0)} \sup_{y\in H'\cap B_{2r}(x_0)}
\left[f(y)-\frac{\|y-z\|_H^2}{2\varepsilon}+\frac{\|x-z\|_H^2}{\varepsilon}\right]\]
is a $C^{1,1}$ convex function on $H'\cap B_r(x_0)$,
$g_\varepsilon$ is $L$-Lipschitz, and $g_\varepsilon\to f$
uniformly on $H'\cap B_r(x_0)$.

By Lemma \ref{quadruple}, the graph of $g_\varepsilon$ over
$H'\cap B_R(x_0)$ satisfies the quadruple condition for
$R=r/3$. The graph of $f$ over $H'\cap B_R(x_0)$ then satisfies
the quadruple condition by continuity.
\end{proof}

\bibliographystyle{plain}
\bibliography{hilbertconvex}

\end{document}